\documentclass[11pt,a4paper,reqno]{amsart}
\usepackage{amsfonts}
\usepackage{amsthm}
\usepackage{amsmath}
\usepackage{hyperref}
\usepackage{amssymb}
\usepackage{amscd}
\usepackage[latin2]{inputenc}
\usepackage{t1enc}
\usepackage[mathscr]{eucal}
\usepackage{indentfirst}
\usepackage{epic}
\numberwithin{equation}{section}
\usepackage[margin=2.9cm]{geometry}
\usepackage{verbatim}
\usepackage{xcolor}
\hypersetup{
    colorlinks,
    linkcolor={red!50!black},
    citecolor={blue!50!black},
    urlcolor={blue!80!black}
}

\theoremstyle{plain}
\newtheorem{Th}{Theorem}[section]

\newtheorem{Prop}[Th]{Proposition}
\newtheorem*{theorem*}{Theorem}

 \theoremstyle{definition}

\newtheorem{Rem}[Th]{Remark}
\newtheorem{?}[Th]{Problem}
\newtheorem{Ex}[Th]{Example}

\begin{document}

\title[Partition Inequalities]{Partition Inequalities and Applications to Sum-Product Conjectures of Kanade-Russell}

\author[W. Bridges]{Walter Bridges}

\address{Louisiana State University \\ Department of Mathematics }

\email{wbridg6@lsu.edu}

 \subjclass[2019]{05A17 (05A20 11P81)}

 \keywords{partition inequalities, Ehrenpreis' problem, anti-telescoping method}

\begin{abstract} We consider differences of one- and two-variable finite products and provide combinatorial proofs of the nonnegativity of certain coefficients.  Since the products may be interpreted as generating functions for certain integer partitions, this amounts to showing a partition inequality.  This extends results due to Berkovich-Garvan and McLaughlin.  We then apply the first inequality and Andrews' Anti-telescoping Method to give a solution to an ``Ehrenpreis Problem'' for recently conjectured sum-product identities of Kanade-Russell.  That is, we provide some evidence for Kanade-Russell's conjectures by showing nonnegativity of coefficients in differences of product-sides as Andrews-Baxter and Kadell did for the product sides of the Rogers-Ramanujan identities.
\end{abstract}

\maketitle

\section{Introduction}

A {\it partition} $\lambda$ of an integer $n$ is a multi-set of positive integers $\{\lambda_1, \dots, \lambda_{\ell} \}$, whose {\it parts} satisfy $$\lambda_1 \geq \lambda_2 \geq \dots \geq \lambda_{\ell} \geq 1 \qquad \text{and} \qquad \sum_{j=1}^{\ell} \lambda_j = n.$$
We will often use {\it frequency notation} to refer to a partition, where $\left(r^{m_r}, \dots, 2^{m_2},1^{m_1} \right)$ represents the partition in which the part $i$ occurs $m_i$ times for $1 \leq i \leq r.$  

Visually, a partition $\lambda$ may be represented by its {\it Ferrer's diagram}, in which parts are displayed as rows of dots.  For example, the Ferrer's diagram of the partition $(5,3,2^2)$ is the array below. $$\begin{matrix} \bullet & \bullet & \bullet & \bullet & \bullet \\ \bullet & \bullet & \bullet & \\ \bullet & \bullet & \\ \bullet & \bullet &  \end{matrix}$$

For two $q$-series $f(q)= \sum_{n \geq 0} a_nq^n$ and $g(q)= \sum_{n \geq 0} b_n q^n$, we write $f(q) \succeq g(q)$ if $a_n \geq b_n$ for all $n$.  If $f(q) \succeq 0$, we will say that $f$ is a {\it nonnegative series}.

We will use the standard $q$-Pochhammer symbol, \begin{align*} &(a;q)_n := \prod_{j=0}^{n-1} \left(1-aq^j \right), \quad (a;q)_{\infty} := \lim_{n \to \infty} (a;q)_n, \quad \text{and} \\ &(a_1, \dots, a_r;q)_n := (a_1;q)_n \cdots (a_r;q)_n. \end{align*}  By convention, an empty product equals 1.

The study of the type of partition inequality we consider began at the 1987 A.M.S. Institute on Theta Functions with a question of Leon Ehrenpreis about the Rogers-Ramanujan Identities (\cite{AP}, Cor. 7.6 and Cor. 7.7):
$$\mathcal{RR}_1: \quad \sum_{n \geq 0} \frac{q^{n^2}}{(q;q)_n} = \frac{1}{(q,q^4;q^5)_{\infty}},$$ $$ \mathcal{RR}_2: \quad \sum_{n \geq 0} \frac{q^{n^2+n}}{(q;q)_n}= \frac{1}{(q^2,q^3;q^5)_{\infty}}.$$
The identity  $\mathcal{RR}_1$ may be interpreted as an equality of certain partition generating functions, giving that the number of partitions of $n$ such that the gap between successive parts is at least 2 equals the number of partitions of $n$ into parts congruent to $\pm 1 \pmod{5}$.  Similarly, $\mathcal{RR}_2$ gives that the number of partitions of $n$ such that the gap between successive parts is at least 2 and 1 does not occur as a part equals the number of partitions of $n$ into parts congruent to $\pm 2 \pmod{5}$.  Thus, both combinatorially and algebraically, it is easy to see that $$ \sum_{n \geq 0} \frac{q^{n^2}}{(q;q)_n}-\sum_{n \geq 0} \frac{q^{n^2+n}}{(q;q)_n} \succeq 0.$$  Therefore, it also holds that \begin{equation}\label{E:RR} \frac{1}{(q,q^4;q^5)_{\infty}} - \frac{1}{(q^2,q^3;q^5)_{\infty}}  \succeq 0. \end{equation}  Ehrenpreis' Problem was to provide a proof of \eqref{E:RR} that did not reference the (heavy-handed and quite nontrivial) Rogers-Ramanujan identities.

Solutions to Ehrenpreis' Problem have been given in various ways.  In the course of proving \eqref{E:RR}, Andrews-Baxter \cite{AB} were led to a new ``motivated'' proof of the Rogers-Ramanujan Identities themselves.  A direct combinatorial proof of \eqref{E:RR} was provided by Kadell \cite{K}, who constructed an injection from the set of partitions of $n$ with parts congruent to $\pm 2 \pmod{5}$ to those with parts congruent to $\pm 1 \pmod{5}$.  Later, Andrews developed the Anti-telescoping Method for showing positivity in differences of products like \eqref{E:RR} \cite{A}.  This method was used by Berkovich-Grizzell in \cite{BG15} to prove infinite classes of partition inequalities, such as the following.

\begin{theorem*}[Theorem 1.2 of \cite{BG15}]
For any octuple of positive integers $(L,m,x,y,z,r,s,u)$, $$\frac{1}{(q^x,q^y,q^z,q^{rx+sy+uz};q^m)_L} - \frac{1}{(q^{rx},q^{sy},q^{uz},q^{x+y+z};q^m)_L} \succeq 0.$$
\end{theorem*}

In \cite{BG05}, Berkovich-Garvan generalized \eqref{E:RR} to an arbitrary modulus as follows.

\begin{theorem*}[Theorem 5.3 of \cite{BG05}] Suppose $L \geq 1$ and $1 \leq r < \frac{M}{2}$.  Then $$\frac{1}{(q,q^{M-1};q^M)_L} - \frac{1}{(q^r,q^{M-r};q^M)_L} \succeq 0$$ if and only if $r \nmid (M-r)$.
\end{theorem*}

One can apply Berkovich-Garvan's result to solve similar ``Ehrenpreis Problems'' for Kanade-Russell's ``mod 9 identities'' in \cite{KR14}.

The first result in this paper extends the above in a way that is independent of the modulus.  We will use this in Section 3 to solve an ``Ehrenpreis Problem'' for conjectural product-sum identities of Kanade-Russell in \cite{KR18}.

\begin{Th}\label{T:BG}
Let $a,b,c$ and $M$ be integers satisfying $1 < a < b < c$ and $1+c=a+b$.
Then if $a \nmid b$, $$\frac{1}{(q,q^c;q^M)_L} - \frac{1}{(q^a,q^b;q^M)_L} \succeq 0 \qquad \text{for any $L \geq 0.$}$$
\end{Th}

Note that we do not necessarily assume $a,b,c \leq M$.  Translated into a partition inequality, Theorem 1.1 says that there are more partitions of $n$ into parts of the forms $Mj+1$ and $Mj+c$ than there are partitions of $n$ into parts of the forms $Mj+a$ and $Mj+b$, where throughout $1 \leq j \leq L$.

Partition inequalities with a fixed number of parts were considered by McLaughlin in \cite{M}.  Answering two of McLaughlin's questions, we give combinatorial proofs of finite analogues of Theorems 7 and 8 from \cite{M}.

\begin{Th}
\label{T:M1}
Let $a,b$ and $M$ be integers satisfying $1 \leq a < b < \frac{M}{2}$ and $\gcd(b,M)=1$.  Define $c(m,n)$ by $$\frac{1}{(zq^a,zq^{M-a};q^M)_{L}(1-q^{LM+a})} - \frac{1}{(zq^b,zq^{M-b};q^M)_{L}} =: \sum_{m,n \geq 0} c(m,n)z^mq^n.$$  Then for any $L,n\geq 0$, we have $c(m,nM) \geq 0$.  If in addition $M$ is even and $a$ is odd, then we also have $c\left(m,nM+\frac{M}{2}\right) \geq 0$ for every $n \geq 0$.
\end{Th}

Note that we do not necessarily make the assumption $\gcd(a,M)=1$ that is in \cite{M}.  While these partition inequalities hold only for $n$ in certain residue classes $\pmod{M}$, Theorem \ref{T:M1} is a strengthening of Theorem \ref{T:BG} for these $n$.  The following is a distinct parts analogue.

\begin{Th}\label{T:M2}
Let $a,b$ and $M$ be integers satisfying $1 \leq a < b < \frac{M}{2}$ and $\gcd(b,M)=1$.  Define $c(m,n)$ by $$(-zq^a,-zq^{M-a};q^M)_{L}\left(1+zq^{LM+a}\right) - (-zq^b,-zq^{M-b};q^M)_{L} =: \sum_{m,n \geq 0} d(m,n)z^mq^n.$$  Then for any $L,n\geq 0$, we have $d(m,nM) \geq 0$.  If in addition $M$ is even and $a$ is odd, then we also have $d\left(m,nM+\frac{M}{2}\right) \geq 0$ for every $n \geq 0$.
\end{Th}

\begin{Rem} Taking the limit as $L \to \infty$ in Theorems \ref{T:M1} and \ref{T:M2} recovers McLaughlin's original partition inequalities. \end{Rem}

In Section 2, we begin by reviewing the $M$-modular diagram of a partition.  Then we provide combinatorial proofs of Theorems \ref{T:BG}-\ref{T:M2}.  In Section 3, we apply Theorem \ref{T:BG} and Andrews' Method of Anti-Telescoping (see \cite{A}) to give a solution to an ``Ehrenpreis Problem'' for recently conjectured sum-product identities of Kanade-Russell \cite{KR18}.  Our concluding remarks in Section 4 ask for Andrews-Baxter style ``motivated proofs'' of these conjectured identities.

\section*{Acknowledgements}
The author thanks Alexander Berkovich and Shashank Kanade for helpful comments on an earlier draft.

\section{Combinatorial Proofs of Theorems}
\subsection{Notation}  The {\it $M$-modular diagram} of a partition $\lambda =\{\lambda_1, \dots, \lambda_{\ell}\}$ is a modification of the Ferrer's diagram, wherein each $\lambda_j$ is first written as $Mq + r$ for $0 \leq r < M,$ and then is represented as a row of $q$ $M$'s and a single $r$ at the end of the row.    These $r$'s we will refer to as {\it ends} or {\it r-ends}.  For example, the $10$-modular diagram of $\lambda= (53^2, 46, 36,16, 11,1)$ has three 6-ends, two 3-ends and two 1-ends:
$$
\begin{matrix} 10 & 10 & 10 & 10 & 10 & 3 \\ 10 & 10 & 10 & 10 & 10 & 3 \\ 10 & 10 & 10 & 10  & 6 \\ 10 & 10 & 10 & 6 \\ 10 & 6 \\ 10 & 1 \\ 1 \\ \end{matrix}
$$

We will also speak of {\it attaching/removing} a column from an $M$-modular diagram.  These operations are best defined with an example:

\begin{align*}
\begin{matrix}
10 & 10 & 10 & 10 & 10 \\
10 & 10 & 10 & 10 & 10 \\
10 & 10 & 10 & 10 & \\
10 & 10 & 10 &  & \\
10 & 10 & 10 &  & \\
10 & 10 & 10 &  & \\
10 & & & & \\
\end{matrix} \underbrace{\cup}_{\substack{\text{{\bf attach}} \\ \text{{\bf the column}}} } \ \begin{matrix} 10 \\  10 \\  10 \\  \end{matrix} \ \ \ &\longrightarrow \ \ \ \begin{matrix}
10 & 10 & 10 & 10 & 10 & 10 \\
10 & 10 & 10 & 10 & 10 & 10\\
10 & 10 & 10 & 10 &  10 & \\
10 & 10 & 10 &  & \uparrow \\
10 & 10 & 10 &  & \\
10 & 10 & 10 &  & \\
10 & & & & \\
\end{matrix}
\\
\begin{matrix}
10 & 10 & 10 & 10 & 10 \\
10 & 10 & 10 & 10 & 10 \\
10 & 10 & 10 & 10 & \\
10 & 10 & 10 &  & \\
10 & 10 & 10 &  & \\
10 & 10 & 10 &  & \\
10 & & \uparrow  & & \\
\end{matrix} \underbrace{\setminus}_{\substack{\text{{\bf remove}} \\ \text{{\bf the column}}}} \ \begin{matrix} 10 \\ 10 \\ 10 \\ 10 \\ 10 \\ 10 \\ \end{matrix} \ \ \ &\longrightarrow  \ \ \ \begin{matrix}
10 &  10 & 10 & 10 \\
10 &  10 & 10 & 10 \\
10 &  10 & 10 & \\
10 &  10 &  & \\
10 &  10 &  & \\
10 &  10 &  & \\
10  & & & \\
\end{matrix}
\end{align*}
We shall only attach or remove columns consisting entirely of $M$'s, and it is easy to see that these operations preserve $M$-modular diagrams.

\subsection{Proof of Theorem \ref{T:BG}}  We provide a combinatorial proof via injection that is nearly identical to that of Theorem 5.1 in \cite{BG05}, but we highlight a technical difference that arises in the general version.  In keeping with \cite{BG05}, we let $\nu_j=\nu_j(\lambda)$ denote the number of parts of $\lambda$ congruent to $j \pmod{M}$.  (The modulus never varies and will be clear from context.)

\begin{proof}

First let $L=1$.  We will prove the general case as a consequence of this one.  For each $n$, we seek an injection $$ \varphi_1: \left\{(a^k,b^{\ell}) \vdash n : k, \ell \geq 0 \right\} \hookrightarrow \left\{(1^k,c^{\ell}) \vdash n : k, \ell \geq 0\right\}.$$    Let $d:=$gcd$(a,b)$.  Explicitly, $\varphi_1$ is as follows. $$\varphi_1\left(a^k,b^{\ell}\right) = \begin{cases} \left(1^{\ell+a(k-\ell)}, c^{\ell} \right) & \text{if $k \geq \ell$,} \hspace{25mm} \qquad \text{(Case 1)} \\ \left(1^{k+b(\ell-k)}, c^{k} \right) & \text{if $\ell > k$ and $\frac{a}{d} \nmid (\ell-k)$,} \qquad \text{(Case 2)} \\ \left(1^{k+1+b(\ell-k-1)-a}, c^{k+1} \right) & \text{if $\ell > k$ and $\frac{a}{d} \mid (\ell-k)$.} \qquad \text{(Case 3)} \\ \end{cases}$$  
This definition can be motivated by noting that each preimage consists either of $k$ pairs $(a,b)$ and $k-\ell$ excess $a$'s, or of $\ell$ pairs $(a,b)$ and $\ell-k$ excess $b$'s.  (There can also be no excess.)  The pairs are mapped as $(a,b) \mapsto (1,c)$.  The excess $a$'s or $b$'s are treated by the following cases.

\begin{flalign*}
\text{Case 1.}  \quad & \text{For the $k-\ell$ excess $a$'s, $(a) \mapsto (1^a)$.} && \\
\text{Case 2.}  \quad & \text{For the $\ell-k$ excess $b$'s, $(b) \mapsto (1^b)$.} && \\
\text{Case 3.}  \quad & \text{For all but the last two excess $b$'s, $(b) \mapsto (1^b)$.  For the last two $b$'s} &&  \\ \quad &\text{$(b) \mapsto (1^{b-a+1},c)$.} && \\
\end{flalign*}
Note that in Case 3 there are at least two excess $b$'s, for if not, $\frac{a}{d}=1$ and then $a \mid b$, \newline a contradiction.  Also, by hypothesis, $b > \frac{c}{2}$, so that $2b > c$.

Let $ \left(1^{\nu_1}, c^{\nu_c} \right)$ be a partition in the image of $\varphi_1$.    The cases are separated as follows:

\begin{flalign*}
\text{Case 1.} \quad & a \mid (\nu_{1}-\nu_{c}), && \\
\text{Case 2.} \quad & a \nmid (\nu_1-\nu_{c}) \ \text{and} \ b \mid (\nu_1-\nu_{c}), && \\
\text{Case 3.} \quad & \nu_1-\nu_{c} \equiv -b \pmod{a} \ \text{and} \ \nu_{1}-\nu_{c} \equiv -a  \pmod{b}. &&
\end{flalign*}

This concludes the proof for $L=1$.

Now let $L \geq 2$.  Again we define an injection
$$ \varphi_L: \left\{ \lambda \vdash n : \lambda_j \in \{ a, b, \dots, LM + a, LM + b\} \right\} \hookrightarrow \left\{ \lambda \vdash n : \lambda_j \in\{ 1, c, \dots, LM + 1, LM + c\} \right\}.$$
Let $\lambda$ be a partition in the left set.  Then $\lambda$ consists of the triple $$\left(\lambda_{(a)},\lambda_{(b)}, (a^k,b^{\ell} ) \right),$$ where $\lambda_{(a)}$ is the $M$-modular diagram obtained by subtracting $a$ from every part of the form $Mj+a$; $\lambda_{(b)}$ is defined similarly.  We apply $\varphi_1$ to $(a^k,b^{\ell})$ and reattach the $1$-ends and $c$-ends based on the case into which $(a^k,b^{\ell})$ falls.
\\
\

 Case 1: $k \geq \ell.$  Attach the $1$-ends to $\lambda_{(a)}$ and the $c$-ends to $\lambda_{(b)}$.  The map $\varphi_1$ guarantees exactly $\#\lambda_{(b)}$ $c$-ends.  Likewise, there are at least as many $1$-ends as there are parts of $\lambda_{(a)}$; any excess 1's are attached as parts to $\lambda_{(a)}$.  The required image of $\lambda$ is then the union of these two partitions.
\\
\

Cases 2 and 3: $\ell > k$.  Attach the $1$-ends to $\lambda_{(b)}$ and the $c$-ends to $\lambda_{(a)}$ as before.  $\varphi_1$ guarantees at least $\# \lambda_{(a)}$ $c$-ends.  In Case 2 we are guaranteed at least $\#\lambda_{(a)}$ 1-ends because $b > 1$ implies $$k + b(\ell-k) > \ell.$$ In Case 3, $\frac{a}{d}>1$ implies $\ell-k > 1$, so $$ k+1 + b(\ell-k-1)-a= \ell + (b-1)(\ell-k-1)-a \geq \ell,$$ and we are guaranteed at least $\#\lambda_{(a)}$ 1-ends.
\\
\

Given the image of $\lambda$, we may clearly recover $\lambda_{(a)}$ and $\lambda_{(b)}$ based on its $1$-ends and $c$-ends and the fact that $\varphi_1$ is an injection.  Thus, $\varphi_L$ is an injection.

\end{proof}

\begin{Rem} The condition $a \nmid b$ in Theorem \ref{T:BG} is necessary to avoid cases like $$\frac{1}{(q,q^5;q^6)_{L}} - \frac{1}{(q^2,q^4;q^6)_L},$$ in which the coefficient of $q^4$ is $-1$. \end{Rem}

\begin{Rem}  If we had copied the proof of Theorem 5.1 in \cite{BG05} exactly, then the conditions ``$\frac{a}{d} \mid$'' and ``$\frac{a}{d} \nmid$'' would be replaced by ``$a \mid$'' and ``$a \nmid$''.  But this is not an injection because Case 2 is only correctly separated from the other two when gcd$(a,b)=1$.  For example, this direct version of Berkovich-Garvan's map gives:
$$\begin{cases} 4^7,6^4 \\ 4^4,6^6 \end{cases} \longrightarrow (1^{16},9^4), \qquad \text{instead of our} \qquad \begin{cases} 4^7,6^4 \\ 4^4,6^6 \end{cases} \longrightarrow \begin{cases} 1^{16},9^4 \\ 1^7,9^5  \end{cases}.$$  In the first example, the partitions fall into cases 1 and 2.  The second example corrects the overlap and places the partitions into cases 1 and 3. \end{Rem}

We demonstrate the injection of Theorem \ref{T:BG} with an example.

\begin{Ex} Here, $(n,M,L,a,b,c)=(52,10,2,4,6,9)$.  Numbers above arrows indicate the case into which a preimage falls.
$$
\begin{matrix}
16^3,4 & \stackrel{3}{\to} & 11^3,9^2,1 \\
16^2,14,6 & \stackrel{3}{\to} & 19,11^2,9,1^2\\
16^2,6^2,4^2 & \stackrel{3}{\to} & 11^2,9^3,1^3\\
16^2,4^5 & \stackrel{1}{\to} & 19^2,1^{14}\\
16,14^2,4^2 & \stackrel{1}{\to} & 19,11^2, 1^{11}\\
16,14,6^3,4 & \stackrel{3}{\to} & 19,11,9^2,1^4\\
16,14,6,4^4 & \stackrel{1}{\to} & 19,11,9,1^{13}\\
16,6^6 & \stackrel{2}{\to} & 11,1^{41}\\
16,6^4,4^3 & \stackrel{3}{\to} & 11,9^4,1^5 \\
16,6^2,4^6 & \stackrel{1}{\to} & 19,9^2,1^{15} \\
16,4^9 & \stackrel{1}{\to} & 19,1^{33}\\
14^3,6,4 & \stackrel{1}{\to} & 11^3, 9, 1^{10}\\ \end{matrix} \qquad \vline \qquad \begin{matrix}
14^2,6^4& \stackrel{3}{\to} & 19^2,9,1^5\\
14^2,6^2,4^3 & \stackrel{1}{\to} & 11^2,9^2,1^{12}\\
14^2,4^6 & \stackrel{1}{\to} & 11^2, 1^{30} \\
14,6^5,4^2 & \stackrel{3}{\to} & 19,9^3,1^6 \\
14,6^3,4^5 & \stackrel{1}{\to} & 11,9^3,1^{14}\\
14,6,4^8 & \stackrel{1}{\to} & 11,9,1^{32} \\
6^8,4 & \stackrel{2}{\to} & 9,1^{43} \\
6^6,4^4 & \stackrel{3}{\to} & 9^5,1^7 \\
6^4,4^7 & \stackrel{1}{\to} & 9^4, 1^{16} \\
6^2,4^{10} & \stackrel{1}{\to} & 9^2,1^{34} \\
4^{13} & \stackrel{1}{\to} & 1^{52} \\
\end{matrix}
$$
\end{Ex}

\subsection{Proofs of Theorems \ref{T:M1} and \ref{T:M2}}  We begin by recalling the main steps in McLaughlin's proof of Theorem 7 from \cite{M}; our proof is based on a combinatorial reading.  First, Cauchy's Theorem (\cite{AP}, Th. 2.1) is used with some algebraic manipulation to write, for fixed $m$, 
\begin{align*} \sum_{n \geq 0} c(m,n)q^n &= \sum_{\substack{0\leq k < \frac{m}{2} \\ M \mid m-2k}} \frac{q^{kM}}{(q^M;q^M)_{m-k}(q^M;q^M)_k} \\ & \qquad \times\left(q^{a(m-2k)} + q^{(M-a)(m-2k)} - q^{b(m-2k)} - q^{(M-b)(m-2k)} \right). \end{align*}  
It then happens that the factor in parentheses is equal to $$ q^{a(m-2k)}\left(1-q^{(b-a)(m-2k)}\right)\left(1-q^{(M-b-a)(m-2k)}\right).$$  But the conditions on $a,b$ and $M$ that lead to the condition $M \mid (m-2k)$ in the sum imply that both factors above are canceled in $\frac{1}{(q^M;q^M)_{m-k}}.$  This gives nonnegativity.

The key steps in the proof are the decomposition of the sum over $k$ and the nonnegativity of $$\frac{(1-q^r)(1-q^s)}{(q;q)_n} \qquad \text{for $1 \leq r < s \leq n$.}$$  Both of these have simple combinatorial explanations, which we employ with $M$-modular diagrams to piece together a proof of Theorem \ref{T:M1}.  Our proof naturally leads to the finite versions with any $L\geq 1$ instead of $\infty$. The proof of Theorem \ref{T:M2} is then a slight modification.

\begin{proof}[Proof of Theorem 1.2]
Let $\mathcal{P}(n,m,j,A)$ denote the set of partitions of $n$ into $m$ parts congruent to $\pm j$ modulo $M$ such that the largest part is at most $A$.  (We have suppressed the modulus $M$ from the notation.)  Let $\mathcal{P}_k(n,m,j,A)$ be the subset of partitions $\lambda \in \mathcal{P}(n,m,j,A)$ with either $\nu_j(\lambda)=k$ or $ \nu_{M-j}(\lambda)= k.$

Clearly, $\mathcal{P}(n,m,j,A) = \bigcup_{0 \leq k \leq \frac{m}{2}} \mathcal{P}_k(n,m,j,A).$  Thus, to show $$\mathcal{P}(nM,m,b,LM-b) \hookrightarrow \mathcal{P}(nM,m,a,LM+a),$$ we may provide injections $$\varphi_k: \mathcal{P}_k(nM,m,b,LM-b) \hookrightarrow \mathcal{P}_k(nM,m,a,LM+a)$$ for each $k \in [0,\frac{m}{2}]$.

Each $\lambda \in \mathcal{P}(nM,m,b,LM-b)$ consists of a triple $$\left(\lambda_{(b)}, \lambda_{(M-b)}, (b^{\nu_b},(M-b)^{\nu_{M-b}}) \right),$$ where $\lambda_{(b)}$ is the $M$-modular diagram with $\nu_b$ nonnegative parts created by removing the $b$-ends.  The $M$-modular diagram $\lambda_{(M-b)}$ is defined analogously by removing the $(M-b)$-ends.

When $k=\frac{m}{2}$, we simply map $$\varphi_{\frac{m}{2}}\left(\lambda_{(b)}, \lambda_{(M-b)}, (b^{\frac{m}{2}},(M-b)^{\frac{m}{2}}) \right) := \left(\lambda_{(b)}, \lambda_{(M-b)}, (a^{\frac{m}{2}},(M-a)^{\frac{m}{2}}) \right).$$  The required partition is then obtained by reattaching the $a$-ends to $\lambda_{(b)}$ and reattaching the $(M-a)$-ends to $\lambda_{(M-b)}$.

Now assume $k < \frac{m}{2}.$  Note that
\begin{equation}\label{nub}
0 \equiv nM \equiv b\nu_b(\lambda)-b\nu_{M-b}(\lambda) \pmod{M}, \end{equation} which implies $\nu_b(\lambda) - \nu_{M-b}(\lambda) \equiv 0 \pmod{M}$ because $\gcd(b,M)=1$.  Thus, we assume without loss of generality that $M \mid (m-2k).$

Let $y:=\frac{(b-a)(m-2k)}{M}$ and $z:= \frac{(M-b-a)(m-2k)}{M}$.  These are positive integers.
\bigskip

\noindent    Case 1: $\nu_{M-b}(\lambda)=k$.  There are $k$ pairs of $(b,M-b)$ and $m-2k$ excess $b$'s.  We map
$$\varphi_k\left(\lambda_{(b)}, \lambda_{(M-b)}, (b^{m-k},(M-b)^{k}) \right) := \left(\lambda_{(b)} \cup  \underbrace{\begin{bmatrix} M \\ \vdots \\ M \end{bmatrix}}_{y \ \text{rows}}, \lambda_{(M-b)}, (a^{m-k},(M-a)^{k}) \right) $$ $$=: \left(\lambda_{(b)}', \lambda_{(M-b)}, (a^{m-k},(M-a)^k \right).$$  Here $\lambda_{(b)}'$ is the $M$-modular diagram formed by attaching the above column to $\lambda_{(b)}$.  Note that $a < b < M$ implies
$0 < y < m-k,$
so that $\lambda_{(b)}'$ is still an $M$-modular diagram with $m-k$ nonnegative parts.

To obtain the required partition, attach the $a$-ends to $\lambda_{(b)}'$ and the $(M-a)$-ends to $\lambda_{(M-b)}$.  It is evident that there are $m$ parts.  Size is preserved, as
\begin{align*}
& |\lambda_{(b)}'| + | \lambda_{(M-b)}| + (m-k)a + k(M-a) \\ &= |\lambda_{(b)}| + My + | \lambda_{(M-b)}| + a(m-2k) + kM \\&= |\lambda_{(b)}| + (b-a)(m-2k) + | \lambda_{(M-b)}| + a(m-2k) + kM  \\ & = |\lambda_{(b)}|  + | \lambda_{(M-b)}|+ b(m-2k) + kM \\ &= |\lambda|.
\end{align*}
Moreover, it is clear that the operations are reversible, so that, within Case 1, \newline $\varphi_k$ is an injection.
\bigskip

\noindent Case 2a: $\nu_b(\lambda)=k$ and $\lambda_{(M-b)}$ does not contain a column of length $y$.\footnote{Or equivalently, the $y$-th part of $\lambda_{(M-b)}$ equals the $(y+1)$-st part.}  There are $k$ pairs of $(b,M-b)$ and $m-2k$ excess $(M-b)$'s.  We map $$\varphi_k\left(\lambda_{(b)}, \lambda_{(M-b)}, (b^{k},(M-b)^{m-k}) \right) := \left(\lambda_{(b)} , \lambda_{(M-b)}\cup  \underbrace{\begin{bmatrix} M \\ \vdots \\ M \end{bmatrix}}_{z \ \text{rows}}, (a^{m-k},(M-a)^{k}) \right) $$ $$=: \left(\lambda_{(b)}, \lambda_{(M-b)}', (a^{m-k},(M-a)^k \right),$$  where $\lambda_{(M-b)}'$ is defined by attaching the above column.  Note again that $b,a < \frac{M}{2}$ implies $0 < z < m-k,$ so that $\lambda_{(M-b)}'$ is still an $M$-modular diagram with $m-k$ nonnegative parts.  Furthermore, $b-a \neq M-b-a$, so $\lambda_{(M-b)}'$ still does not contain a column of length $y.$

To obtain the required partition, attach the $a$-ends to $\lambda_{(M-b)}'$ and the $(M-a)$-ends to $\lambda_{(b)}$.  It is evident that there are $m$ parts.  Size is preserved, as
\begin{align*}
& |\lambda_{(b)}| + | \lambda_{(M-b)}'| + (m-k)a + k(M-a) \\ &= |\lambda_{(b)}|  + | \lambda_{(M-b)}| + Mz+ a(m-2k) + kM\\&= |\lambda_{(b)}|  + | \lambda_{(M-b)}| + (M-b-a)(m-2k)+ a(m-2k) + kM  \\ & = |\lambda_{(b)}|  + | \lambda_{(M-b)}|+ (M-b)(m-2k) + kM \\ &= |\lambda|.
\end{align*}
Moreover, it is clear that the operations are reversible, so that, within Case 2a, $\varphi_k$ is an injection.
\bigskip

\noindent  Case 2b: $\nu_b(\lambda)=k$ and $\lambda_{(M-b)}$ contains a column of length $y$.\footnote{Or equivalently, the $y$-th part of $\lambda_{(M-b)}$ is strictly greater than the $(y+1)$-st part.}  In this case we send $$\left(\lambda_{(b)}, \lambda_{(M-b)}, (b^{k},(M-b)^{m-k}) \right) \mapsto \left(\lambda_{(b)} , \lambda_{(M-b)}\setminus \underbrace{\begin{bmatrix} M \\ \vdots \\ M \end{bmatrix}}_{y \ \text{rows}}, (a^{k},(M-a)^{m-k}) \right) $$ $$=: \left(\lambda_{(b)}, \lambda_{(M-b)}', (a^{k},(M-a)^{m-k} \right),$$ where  $\lambda_{(M-b)}'$ is defined by removing the above column.  As before, we still may consider $\lambda_{(M-b)}'$ an $M$-modular diagram with $m-k$ nonnegative parts.

To obtain the required partition, attach the $a$-ends to $\lambda_{(b)}$ and the $(M-a)$-ends to $\lambda_{(M-b)}'$.  It is evident that there are $m$ parts.  Size is preserved, as
\begin{align*}
 & |\lambda_{(b)}| + |\lambda_{(M-b)}'| + ka + (m-k)(M-a) \\ &= |\lambda_{(b)}|   + | \lambda_{(M-b)}|- My + kM + (M-a)(m-2k)\\ &= |\lambda_{(b)}|   + | \lambda_{(M-b)}|- (b-a)(m-2k) + kM + (M-a)(m-2k)  \\ & = |\lambda_{(b)}|  + | \lambda_{(M-b)}|+ (M-b)(m-2k) + kM \\ &= |\lambda|.
\end{align*}
Moreover, it is clear that the operations are reversible, so that, within Case 2b, $\varphi_k$ is an injection.

Let $\left(\lambda_{(a)},\lambda_{(M-a)}, a^{\nu_a},(M-a)^{\nu_{M-a}}\right)$ lie in the image of $\varphi_k$.  Then cases are separated as follows.

\begin{flalign*} &\text{Case 1: } \quad \nu_a > \nu_{M-a} \ \text{and $\lambda_{(a)}$ contains a column of length $y$.} &&  \\ & \text{Case 2a:} \quad \nu_a > \nu_{M-a} \ \text{and $\lambda_{(a)}$ does not contain a column of length $y$.} && \\ &\text{Case 2b:} \quad \nu_a < \nu_{M-a}. && \\ \end{flalign*} 

Finally, note that in each case $\varphi_k$ adds at most $M$ to the largest part of what becomes $\lambda_{(a)}$, so indeed $\varphi_k$ maps $\mathcal{P}_k(nM,m,b,LM-b)$ into $\mathcal{P}_k(nM,m,a,LM+a)$ as required.  This completes the proof of the first statement.

When $M$ is even and $a$ is odd, we can use exactly the same injections, assuming because of \eqref{nub} that $m-2k \equiv \frac{M}{2} \pmod{M}$.  We note that $\gcd(b,M)=1$ implies that $b$ is also odd, so $y$ and $z$ are still integers.
\end{proof}

\begin{Rem} We note that the extra factor $\frac{1}{(1-q^{LM+a})}$ in the left term of Theorem \ref{T:M1} is necessary.  For example, in $$\frac{1}{(zq^2,zq^5;q^{7})_2}- \frac{1}{(zq^3,zq^4;q^{7})_2},$$ the coefficients of $z^{7}q^{70},z^{13}q^{70},z^{16}q^{70},$ and $z^{18}q^{70}$ are all negative.\end{Rem}

The proof of Theorem \ref{T:M2} is similar, but now cases are determined by columns that occur twice.

\begin{proof}[Proof of Theorem 1.3]
We define injections $\varphi_k'$ to be the same as $\varphi_k$, except that in Cases 2a and 2b we condition on whether or not a partition contains {\it two} columns of length $y$.  This ensures that $\varphi_k'$ preserves distinct parts partitions:

\bigskip
\noindent  Case 1.  Note that $\lambda_{(b)}$ is a distinct parts partition into $m-k$ nonnegative parts (so 0 occurs at most once).  As such, $\lambda_{(b)}$ must contain a column of length $y$.  (Recall that $y < m-k.$) Attaching another such column means that $\lambda_{(b)}'$ still has distinct nonnegative parts.  Attaching the ends as above also preserves distinct parts.

\bigskip
\noindent  Case 2a.  Again attaching the column to $\lambda_{(M-b)}$ preserves distinct parts because $z < m-k$.  The fact that $M-b-a \neq b-a$ implies that $\lambda_{(M-b)}'$ still does not contain two columns of length $y$.

\bigskip
\noindent   Case 2b.  Since $\lambda_{(M-b)}$ contains two columns of length $y$, removing one such column preserves distinct parts.

\bigskip
Cases are separated as follows.

\begin{flalign*} &\text{Case 1: } \quad \nu_a > \nu_{M-a} \ \text{and $\lambda_{(a)}$ contains two columns of length $y$.} &&  \\ & \text{Case 2a:} \quad \nu_a > \nu_{M-a} \ \text{and $\lambda_{(a)}$ does not contain two columns of length $y$.} && \\ &\text{Case 2b:} \quad \nu_a < \nu_{M-a}. && \\ \end{flalign*}
This concludes the proof.
\end{proof}

\begin{Rem}
Unlike in Theorem \ref{T:M1}, it appears that the extra factor $\left(1+q^{LM+a}\right)$ in the left term of Theorem \ref{T:M2} is often not needed for nonnegativity.  A computational search up to $M \leq 12$, $L \leq 20$ and $nM \leq 250$ reveals that for $$\sum_{m,n \geq 0} d'(m,n)z^mq^n := (-zq^a,-zq^{M-a};q^M)_L -(-zq^b,-zq^{M-b};q^M)_L,$$ we have some $d'(m,nM) < 0$ only when $(a,b,M)=(1,2,5)$.
\end{Rem}

In fact, we can condition on more than just 2 columns to prove the following new result.

\begin{Prop}\label{P:Md}
Let $d \geq 0$, $1\leq a< b < \frac{M}{2}$ and $\gcd(b,M)=1$.  Let $p^{(d)}(n,m,j,A)$ denote the number of partitions of $n$ into $m$ parts congruent to $\pm j \pmod{M}$, whose parts are at most $A$ such that the gap between successive parts is greater than $dM$.  Then for all $n ,m \geq 0,$ $$p^{(d)}(nM,m,a,LM+a) \geq p^{(d)}(nM,m,b,LM-b).$$  If in addition $a$ is odd, then we also have $$p^{(d)}\left(nM+\frac{M}{2},m,a,LM+a\right) \geq p^{(d)}\left(nM+\frac{M}{2},m,b,LM-b\right).$$
\end{Prop}

Substituting $d=0$ and $d=1$ gives Theorems \ref{T:M1} and \ref{T:M2} respectively.
\begin{proof}
Let $\lambda = \left(\lambda_{(b)}, \lambda_{(M-b)}, (b^{\nu_b}, (M-b)^{\nu_{M-b}}) \right)$ be a partition counted by $p^{(d)}(nM,m,b,LM-b)$.  Then the $M$ modular diagrams $\lambda_{(b)}$ and $\lambda_{(M-b)}$ are partitions into nonnegative multiples on $M$ such that the difference in successive parts is at least $ (d+1)M$.  Our injections $\varphi_k^{(d)}$ are the same as before, except that we condition in cases 2 or 3 on whether or not $\lambda_{(M-b)}$ contains $d+2$ columns of length $y$.
\end{proof}

\section{Applications to Kanade-Russell's Conjectures}
In \cite{KR18}, Kanade-Russell conjectured several new Rogers-Ramanujan-type product-sum identities---three arising from the theory of Affine Lie Algebras, and several companions.  Bringmann--Jennings-Shaffer--Mahlburg were able to prove many of these \cite{BJSM}, and they reduced the sum-sides of the four conjectures below from triple series to a single series.  Here, $\mathcal{KR}_j$ is Identity $j$ in \cite{KR18}, and $H_j(x)$ is the sum side as denoted in \cite{BJSM}.

 \begin{align*} &\mathcal{KR}_4: & H_4(1)&= \frac{1}{(q,q^4,q^5,q^9,q^{11};q^{12})_{\infty}}, \\ &\mathcal{KR}_{4a}: & H_5(1) &= \frac{1}{(q,q^5,q^7,q^8,q^9;q^{12})_{\infty}}, \\ &\mathcal{KR}_{6}: & H_8(1)&= \frac{1}{(q,q^3,q^7,q^8,q^{11};q^{12})_{\infty}}, \\ &\mathcal{KR}_{6a}: & H_9(1)&= \frac{1}{(q^3,q^4,q^5,q^7,q^{11};q^{12})_{\infty}}. \end{align*}

The pairs of sum-sides, $\left(H_4(1),H_5(1) \right)$ and $\left(H_8(1),H_9(1) \right)$, are composed of two generating functions for partitions that satisfy the same set of gap conditions, but $H_5$ and $H_9$ have an additional condition on the smallest part (see \cite{KR18}).  Hence, as with the Rogers-Ramanujan sum-sides, we have $$H_4(1)-H_5(1) \succeq 0 \qquad \text{and} \qquad H_8(1)-H_9(1) \succeq 0,$$ which, {\it if the conjectures are true}, implies the following result.

\begin{Prop}  The following inequalities hold.
\begin{equation}\label{KR1} \frac{1}{(q,q^4,q^5,q^9,q^{11};q^{12})_{\infty}} - \frac{1}{(q,q^5,q^7,q^8,q^9;q^{12})_{\infty}} \succeq 0, \end{equation} \begin{equation}\label{KR2} \frac{1}{(q,q^3,q^7,q^8,q^{11};q^{12})_{\infty}} - \frac{1}{(q^3,q^4,q^5,q^7,q^{11};q^{12})_{\infty}} \succeq 0. \end{equation}
\end{Prop}

\begin{proof}
\eqref{KR2} is an immediate consequence of Theorem \ref{T:BG}, since for every $L \geq 0$, $$\frac{1}{(q,q^8;q^{12})_L} - \frac{1}{(q^4,q^5;q^{12})_L} \succeq 0.$$  Multiplying both sides by the positive series $\frac{1}{(q^3,q^7,q^{11};q^{12})_{\infty}}$ and taking the limit $L \to \infty$ finishes the proof of \eqref{KR2}.

Andrews' Anti-telescoping Method \cite{A} works seamlessly to show \eqref{KR1}.  Define $$P(j):= (q,q^4,q^{11};q^{12})_j \qquad \text{and} \qquad Q(j):= (q,q^7,q^8;q^{12})_j,$$ and note that the following implies \eqref{KR1}: \begin{equation}\label{3.3} \frac{1}{P(L)} - \frac{1}{Q(L)} \succeq 0 \quad \text{for all $L \geq 0$.} \end{equation}  Now we write
\begin{align*} \frac{1}{P(L)}- \frac{1}{Q(L)} &= \frac{1}{Q(L)} \left(\frac{Q(L)}{P(L)} - 1 \right) \\ &= \frac{1}{Q(L)}\sum_{j=1}^L \left(\frac{Q(j)}{P(j)} - \frac{Q(j-1)}{P(j-1)} \right) \\ &= \sum_{j=1}^L \frac{1}{\frac{Q(L)}{Q(j-1)} P(j)} \left(\frac{Q(j)}{Q(j-1)} - \frac{P(j)}{P(j-1)} \right),  \end{align*} whose $j$-th term is $$\frac{(1-q^{12j-11})q^{12(j-1)}}{(q^{12j-11},q^{12j-5},q^{12j-4};q^{12})_{L-j+1}(q,q^4,q^{11};q^{12})_j} \times \left(-q^7-q^8+q^4+q^{11} \right) $$ \begin{equation}\label{jthterm}= \frac{(1-q^{12j-11})q^{12(j-1)}}{(q^{12j-11},q^{12j-5},q^{12j-4};q^{12})_{L-j+1}(q,q^4,q^{11};q^{12})_j} \times q^4(1-q^3)(1-q^4).\end{equation}  The terms $(1-q^4)$ and $(1-q^{12j-11})$ are cancelled in the denominator, and we can write $\frac{1-q^3}{1-q} = 1+q+q^2$.  Hence, \eqref{jthterm} is nonnegative for every $j$, proving \eqref{3.3} and then \eqref{KR1}.
\end{proof}

There is another pair of identities in \cite{KR18} with an Ehrenpreis Problem set-up:
\begin{align*}
&\mathcal{KR}_5: & H_6(1)&= \frac{1}{(q^2;q^4)_{\infty}} \prod_{n \geq 0} \left(1+q^{4n+1} + q^{2(4n+1)} \right), \\ &\mathcal{KR}_{5a}: & H_7(1)&= \frac{1}{(q^2;q^4)_{\infty}} \prod_{n \geq 0} \left(1+q^{4n+3} + q^{2(4n+3)} \right).
\end{align*}

But both identities were proved in \cite{BJSM}, and there is an obvious injection proving $$\frac{1}{(q^2;q^4)_{\infty}} \prod_{n \geq 0} \left(1+q^{4n+1} + q^{2(4n+1)} \right)-\frac{1}{(q^2;q^4)_{\infty}} \prod_{n \geq 0} \left(1+q^{4n+3} + q^{2(4n+3)} \right) \succeq 0,$$ namely, sending each $(4n+3)$ to the pair $(4n+1,2).$

\section{Concluding Remarks}

As we pointed out in the introduction, \eqref{E:RR} was the start of Andrews-Baxter's ``motivated'' proof of the Rogers-Ramanujan identities \cite{AB}.  They defined $G_1:= (q,q^4;q^5)_{\infty}^{-1}$ and $G_2:= (q^2,q^3;q^5)_{\infty}^{-1}$, and then recursively \begin{equation}\label{RRrec} G_i:= \frac{G_{i-2} - G_{i-1}}{q^{i-2}}, \qquad \text{for $i \geq 3$}. \end{equation}  They then observed computationally that $G_i = 1 + \sum_{n \geq i} g_{i,n}q^n \succeq 0.$  Thus, as $i \to \infty$ the coefficient of $q^n$ in $G_i$ is eventually 0.  This was their ``Empirical Hypothesis,'' and proving it leads easily to a new proof of the Rogers-Ramanujan identities.

Note that, starting from the sum-sides of $G_1$ and $G_2$, the recursive definition \eqref{RRrec} and the Empirical Hypothesis are completely natural.  For example, if $\mathcal{RR}$ denotes the set of gap-2 partitions, then by the Rogers-Ramanujan Identities, $$G_1-G_2 = \sum_{\substack{\lambda \in \mathcal{RR} \\ \lambda \ni 1}} q^{|\lambda|} = q\left(1+\sum_{\substack{\lambda \in \mathcal{RR} \\ \lambda_j \geq 3}} q^{|\lambda|} \right),$$ and so $$G_2-G_3 = \sum_{\substack{\lambda \in \mathcal{RR} \\ \lambda_j \geq 2 \\ \lambda_j \ni 2}} q^{|\lambda|} = q^2 \left(1+ \sum_{\substack{\lambda \in \mathcal{RR} \\ \lambda_j \geq 4}} q^{|\lambda|} \right),$$ and so on.

For $\mathcal{KR}_{4}$, $\mathcal{KR}_{4a}$, $\mathcal{KR}_{6}$ and $\mathcal{KR}_{6a}$, the conditions on the sum-sides are complicated, so a recurrence (or set of recurrences) like \eqref{RRrec} that result in an ``Empirical Hypothesis'' may be more difficult to find.  Nevertheless, these ideas have been further developed in \cite{CKLMQRS}, \cite{KLRS} and \cite{LZ} to give ``motivated proofs'' of sum-product identities with gap-conditions that are more complicated than those of $\mathcal{RR}$.  Perhaps further developments will give an answer to the following question: Do there exist ``motivated proofs'' of $\mathcal{KR}_{4}$, $\mathcal{KR}_{4a}$, $\mathcal{KR}_{6}$ and $\mathcal{KR}_{6a}$?  This would be especially interesting, since to our knowledge there have not yet been ``motivated proofs'' featuring asymmetric products.

\end{document}